\documentclass[11pt]{article}

\usepackage{enumerate, amsmath, amssymb,latexsym}
\usepackage{amsthm, color, graphicx, hyperref, amscd}
\usepackage[T1]{fontenc}
\usepackage[utf8]{inputenc}
\usepackage[american]{babel}
\usepackage{lmodern}
\usepackage{microtype}
\usepackage{mathtools}
\usepackage{amssymb}
\usepackage{booktabs}
\usepackage{enumitem}
\usepackage{xcolor}
\usepackage[nameinlink,noabbrev]{cleveref}
\usepackage{todonotes, comment}
\newtheorem{theorem}{Theorem}[section]
\newtheorem{proposition}[theorem]{Proposition}
\newtheorem{lemma}[theorem]{Lemma}
\newtheorem{corollary}[theorem]{Corollary}
\newtheorem{definition}[theorem]{Definition}

\newtheorem{remark}[theorem]{Remark}

\newtheorem{problem}[theorem]{\bf Problem}

\newcommand{\Q}{Q}
\newcommand{\F}{\mathbb{F}}
\newcommand{\isect}{\iota}
\newcommand{\1}{\mathbf{1}}
\newcommand{\R}{\mathbb{R}}
\newcommand{\Z}{\mathbb{Z}}

\newcommand\cH{{\mathcal H}}

\title{Small Boolean Sections in Alon–Füredi Covers}
\author{Zolt\'an L\'or\'ant Nagy\thanks{
		E\"otv\"os Lor\'and University, Budapest, Hungary. The author is supported by the J\'anos Bolyai Scholarship of the Hungarian Academy of Sciences and by the NRDI EXCELLENCE-24 grant no. 151504 Combinatorics and Geometry. Email: zoltan.lorant.nagy@ttk.elte.hu}}
\date{}

\begin{document}
	
	\maketitle
	\begin{abstract}
		Alon and F\"uredi proved that at least $n$ affine hyperplanes are required to cover $\{0,1\}^n\setminus\{\textbf{0}\}$ while avoiding the origin, and that this bound is sharp. We study how small the largest Boolean intersection among the hyperplanes can be in a cover attaining this minimum. Let $F(n)$ denote the minimum possible value of $\max_{H\in\mathcal H} |H\cap\{0,1\}^n|$ over all families $\cH$ of $n$ affine hyperplanes covering $\{0,1\}^n\setminus{\mathbf{0}}$ and avoiding the origin. 
		We give an explicit construction, proving that 
		$F(n)=(1+o(1))\frac{2^n}{n},$
		and hence asymptotically attain the averaging lower bound.
	\end{abstract}
	
	\maketitle
	
	\section{Introduction}
	
	Let $\Q_n=\{0,1\}^n$ denote the $n$-dimensional  hypercube, and let $\Q_n^*=\Q_n\setminus\{\textbf{0}\}$. A classical theorem of Alon and F\"uredi~\cite{AlonFuredi1993} states that every family of affine hyperplanes covering $\Q_n^*$ while avoiding the origin $\textbf{0}$ contains at least $n$ hyperplanes. 
	Equality is attained both by the coordinate hyperplanes
	\[
	x_i=1,\qquad i\in[n],
	\]
	and by the Hamming layers
	\[
	x_1+\cdots+x_n=j,\qquad 1\le j\le n.
	\]
	The constructions differ very much in nature. 
	The first equality construction operates with hyperplanes containing $2^{n-1}$ hypercube vertices, and  for \(n\geq2\), every two distinct coordinate hyperplanes meet in
	\(2^{n-2}\) Boolean vertices. By contrast, the Hamming layers partition
	\(Q_n^*\), but their Boolean intersection sizes vary considerably.  Its largest hyperplane intersection  contains
	\[
	\binom{n}{\lfloor n/2\rfloor}
	=\left(\sqrt{\frac{2}{\pi}}+o(1)\right)\frac{2^n}{\sqrt n}
	\]
	vertices.
	
	This motivates the following minimax refinement of the Alon–Füredi problem: among equality cases using exactly $n$ hyperplanes, minimize the largest Boolean intersection of an individual hyperplane.
	
	\begin{definition}\label{def:F}
		For an affine hyperplane $H\subseteq\R^n$, define its \emph{Boolean intersection size} by
		\[
		\isect_n(H)=|H\cap\Q_n|.
		\]
		Let $F(n)$ be the minimum value of
		\[
		\max_{H\in\mathcal H}\isect_n(H)
		\]
		over all families $\mathcal H$ of exactly $n$ affine hyperplanes that cover $\Q_n^*$ and avoid $\textbf{0}$.
	\end{definition}
	
	\begin{problem}
		Determine the minimum possible largest intersection size  $F(n)=
		\min \max_{H\in\mathcal H}\isect_n(H)$ that can be attained in a cover $\Q_n^*$ by  a family  $\cH$ of $n$ affine hyperplanes which all  avoid the origin $\textbf{0}$.
	\end{problem}
	
	Counting incidences gives the immediate lower bound
	\begin{equation}\label{eq:averaging-lb}
		F(n)\ge \left\lceil\frac{2^n-1}{n}\right\rceil.
	\end{equation}
	Our main result shows that this elementary bound is asymptotically sharp.
	
	\begin{theorem}\label{thm:main}
		We have
		\[
		F(n)=(1+o(1))\frac{2^n}{n}.
		\]
		Equivalently,
		\[
		\lim_{n\to\infty}\frac{nF(n)}{2^n}=1.
		\]
	\end{theorem}
	
	The Alon--F\"uredi theorem was proved using polynomial methods and was one of the forerunners  of the  Combinatorial Nullstellensatz of Alon, whose applications include additive number theory, graph coloring, graph theory, and many other topics in  combinatorics \cite{Alon1999}.
	
	The Combinatorial Nullstellensatz has proved particularly effective in extremal problems where a sharp numerical bound is accompanied by a strong inverse theory. A fundamental example is the Cauchy--Davenport theorem: for nonempty sets $A,B\subseteq  \F_p$, one has
	$|A+B|\geq \min{p,|A|+|B|-1},$
	and the bound admits a short proof by the Combinatorial Nullstellensatz \cite{Alon1999}. Moreover, in the nonsaturated equality case, Vosper's theorem shows that the extremal pairs are essentially forced: if $|A|,|B|\geq 2$ and
	$
	|A+B|=|A|+|B|-1\leq p-2,$
	then $A$ and $B$ are arithmetic progressions with a common difference \cite{Vosper1956}. Related inverse results describe pairs whose sumsets are only slightly larger than the minimum \cite{HamidouneRodseth2000}.
	A similar phenomenon occurs for the Erdős–Ginzburg–Ziv theorem,
	which gives the sharp bound on the length forcing a zero-sum subsequence of prescribed size \cite{ErdosGinzburgZiv1961,GaoGeroldinger2006}.
	Its extremal sequences admit a complete inverse description: Bialostocki and Dierker proved that every sequence of $2n-2$ elements of the cyclic group $C_
	n$ containing no $n$-term zero-sum consists of $n-1$ copies of each of two elements whose difference generates $C_n$ \cite{BialostockiDierker1992}.
	
	The situation for the Alon--Füredi theorem is markedly different. Although the lower bound of $n$ hyperplanes is sharp, the equality condition does not appear to determine the geometry of the covering hyperplanes: even the coordinate and Hamming-layer covers have very different intersection and overlap profiles.
	Thus, no comparable inverse theorem  was known for equality in the Alon--F\"uredi covering theorem.  
	The aim of this paper is to determine, with the sharp asymptotic constant, the smallest possible largest Boolean intersection among all Alon–Füredi equality covers, thereby showing that the  covers may have radically different intersection profiles in the extremal case of $n$ hyperplanes.
	
	Our construction is completely explicit and uses only integer coefficients. At a high level, we write $n=r+m$, distinguish $r$ \textit{control coordinates},  and   for each
	$\textbf{y}\in\{0,1\}^r$, we call ${\textbf{y}}\times\{0,1\}^m$
	the corresponding control slice. Let
	\[
	S=x_{r+1}+\cdots+x_n
	\]
	be the Hamming weight of the remaining $m$ coordinates. A hyperplane of the form
	\[
	a_1x_1+\cdots+a_rx_r+S=d
	\]
	selects one Hamming layer on each of the $2^r$ control slices. We choose the coefficients so that these layer indices form an arithmetic progression as the control slice varies. For a fixed hyperplane, its Boolean intersection size is then a sum of binomial coefficients in one residue class. Choosing $2^r\asymp\sqrt n\log n$ makes these residue classes asymptotically uniform and forces every main intersection size to be $(1+o(1))2^n/n$.
	
	The hyperplane-cover literature has developed in several related directions. Covers avoiding more than one prescribed cube vertex were studied by Aaronson et al.~\cite{AaronsonEtAl2021}. Higher-multiplicity almost covers were investigated by Clifton and Huang~\cite{CliftonHuang2020}, using the punctured Combinatorial Nullstellensatz of Ball and Serra~\cite{BallSerra2009}; multiplicity and layer variants were developed further by Ghosh, Kayal, and Nandi~\cite{GhoshKayalNandi2023}, and a stability theory for the relevant weighted inequality was obtained by Das et al.~\cite{DasEtAl2023}. Covers of symmetric subsets and low-weight complements were considered by Venkitesh~\cite{Venkitesh2022} and Sziklai and Weiner~\cite{SziklaiWeiner2023}. Essential and nondegenerate covers impose different geometric conditions on the coefficient supports \cite{LinialRadhakrishnan2005,SauermannXu2025,SauermannXu2026}. The parameter $F(n)$ concerns a distinct extremal feature: the number of hyperplanes is fixed at the Alon--F\"uredi minimum, while the objective is to minimize the largest Boolean intersection size.
	
	The Alon--Füredi almost-cover problem has also been studied for vertex sets of other polytopes. Hegedüs and Károlyi proved that every almost cover of the vertices of the permutohedron $\Pi_{n-1}$ contains at least $\binom{n}{2}$ affine hyperplanes, and that this bound is sharp \cite{HegedusKarolyi2024}, see also \cite{Karolyi2024}. Their polynomial argument closely parallels the proof of the Alon--Füredi theorem. They also considered covers of the entire vertex set by hyperplanes distinct from its affine hull, giving economical constructions whose intersections with the vertex set are pairwise disjoint and form exact partitions. This is particularly relevant to the present problem: it shows that, beyond the Boolean cube, the number of covering hyperplanes, the sizes of their individual intersections, and the overlap pattern of the cover arise naturally as related but distinct extremal parameters.
	
	
	The paper is organized as follows. In \Cref{sec:intersection-formula} we record a general coefficient formula for the Boolean intersection size of a hyperplane. In \Cref{sec:construction} we give an exact construction for arbitrary parameters and compute all intersection sizes. In \Cref{sec:residues} we establish the required binomial residue estimates and complete the proof of \Cref{thm:main}. Some remarks on the problem and the case of small dimensions are discussed in \Cref{sec:concluding}.
	
	\section{Boolean intersection sizes from the coefficients}\label{sec:intersection-formula}
	
	The following elementary lemma determines the Boolean intersection size of a hyperplane directly from its coefficients. It will also justify all subsequent counting formulas.
	
	\begin{lemma}[Coefficient formula]\label{lem:coefficient-formula}
		Let
		\[
		H(a,b)=\left\{x\in\R^n:\sum_{i=1}^n a_ix_i=b\right\},
		\qquad a=(a_1,\dots,a_n)\ne 0.
		\]
		Then
		\begin{equation}\label{eq:subset-sum-formula}
			\isect_n(H(a,b))
			=\sum_{J\subseteq[n]}\1\!\left\{\sum_{j\in J}a_j=b\right\}.
		\end{equation}
		If the distinct coefficient values are $c_1,\dots,c_u$, with respective multiplicities $m_1,\dots,m_u$, then equivalently
		\begin{equation}\label{eq:grouped-formula}
			\isect_n(H(a,b))
			=\sum_{\substack{0\le k_j\le m_j\ (1\le j\le u)\\
					c_1k_1+\cdots+c_uk_u=b}}
			\prod_{j=1}^u\binom{m_j}{k_j}.
		\end{equation}
		When $a_1,\dots,a_n,b\in\Z$, this is the Laurent-polynomial coefficient
		\begin{equation}\label{eq:generating-function-formula}
			\isect_n(H(a,b))=[z^b]\prod_{i=1}^n(1+z^{a_i}).
		\end{equation}
		The same coefficient formula applies to rational coefficients after clearing denominators.
	\end{lemma}
	
	\begin{proof}
		Each Boolean vector $x\in\Q_n$ is the indicator vector of a unique subset $J\subseteq[n]$, and the hyperplane equation becomes $\sum_{j\in J}a_j=b$. This proves \eqref{eq:subset-sum-formula}. Grouping subsets according to how many indices are chosen from each coefficient class gives \eqref{eq:grouped-formula}. Finally, expanding $\prod_i(1+z^{a_i})$ chooses either $1$ or $z^{a_i}$ from each factor, so the exponent records the corresponding subset sum, proving \eqref{eq:generating-function-formula}.
	\end{proof}
	
	\begin{corollary}\label{cor:control-tail}
		Let $n=r+m$, write $\textbf{x}=(\textbf{y,z})\in\Q_r\times\Q_m$, and put $S(\textbf{z})=z_1+\cdots+z_m$. For
		\[
		H=\left\{(y,z):\alpha_1y_1+\cdots+\alpha_ry_r+S(z)=d\right\},
		\]
		one has
		\begin{equation}\label{eq:control-tail-formula}
			\isect_n(H)=\sum_{\textbf{y}\in\Q_r}\binom{m}{d-\alpha\cdot y},
		\end{equation}
		where $\binom{m}{k}=0$ when $k\notin\{0,\dots,m\}$. Here we adopt the convention that
		\(\binom{m}{u}=0\) unless \(u\in\{0,1,\ldots,m\}\).
	\end{corollary}
	
	\begin{proof}
		For each fixed control vector $\textbf{y}$, the equation is $S(z)=d-\alpha\cdot y$, which has $\binom{m}{d-\alpha\cdot y}$ Boolean solutions.
	\end{proof}
	
	\section{The control-slice construction}\label{sec:construction}
	
	Fix integers $n$ and $r$ with $1\le r<n$, and let $m=n-r$.
	Assume $2^r+r\le n$, so that the integer $s$ defined by $s:= \lfloor\frac{m+1}{2^r+1}\rfloor$ is positive.  Let $t$ be the corresponding remainder, i.e., $m+1=(2^r+1)s+t,\qquad 0\le t<2^r+1.$ 
	Thus 
	\[
	s=\left\lfloor\frac{m+1}{2^r+1}\right\rfloor.
	\]
	We write a  point $\textbf{x}$ of $\Q_n$ as $(\textbf{y},\textbf{z})\in\Q_r\times\Q_m$, where
	\[
	\textbf{y}=(x_1,\dots,x_r),\qquad \textbf{z}=(x_{r+1},\dots,x_n),
	\]
	and set
	\[
	S:=S(z)=x_{r+1}+\cdots+x_n.
	\]
	For $\textbf{b}=(b_1,\dots,b_r)\in\Q_r$, define its \textit{binary value} by
	\[
	[\textbf{b}]=\sum_{i=1}^r2^{i-1}b_i.
	\]
	We use $\textbf{b}\oplus \textbf{y}$ for coordinate-wise addition modulo two. The identity
	\begin{equation}\label{eq:xor}
		[\textbf{b}\oplus \textbf{y}]=[\textbf{b}]+\sum_{i=1}^r2^{i-1}(1-2b_i)y_i
	\end{equation}
	encodes the permutation of Hamming levels across the control slices.
	
	We now define four types of hyperplanes.
	\begin{definition}[Covering hyperplanes] \
		
		\noindent\textbf{I. Main hyperplanes.}
		For $0\le a<s$ and $\textbf{b}\in\Q_r$, let
		\begin{equation}\label{eq:main-hyperplane}
			H_{a,\textbf{b}}:\quad
			S-s\sum_{i=1}^r2^{i-1}(1-2b_i)x_i
			=a+s(1+[\textbf{b}]).
		\end{equation}

		\medskip
		\noindent\textbf{II. Lower Hamming layers.}
		For $1\le j<s$, let
		\begin{equation}\label{eq:lower-layer}
			L_j:\quad S=j.
		\end{equation}

		\medskip
		\noindent\textbf{III. Upper remainder layers.}
		For $0\le\ell<t$, let
		\begin{equation}\label{eq:upper-layer}
			U_\ell:\quad S=(2^r+1)s+\ell.
		\end{equation}

		\medskip
		\noindent\textbf{IV. Cleanup hyperplanes.}
		For $1\le i\le r$, let
		\begin{equation}\label{eq:cleanup}
			C_i:\quad x_i+S=1.
		\end{equation}
		
	\end{definition}
	\medskip

	\begin{theorem}[Exact construction]\label{thm:exact-construction}

		Let
		\begin{align*}
			\mathcal H(n,r)
			={}&\{H_{a,\textbf{b}}:0\le a<s,\ \textbf{b}\in\Q_r\}
			\,\cup\,\{L_j:1\le j<s\}\\
			&\cup\,\{U_\ell:0\le\ell<t\}
			\,\cup\,\{C_i:1\le i\le r\}.
		\end{align*}
		The family $\mathcal H(n,r)$ consists of exactly $n$ affine hyperplanes, covers $\Q_n^*$, and avoids the origin.
	\end{theorem}
	
	\begin{proof}
		The number of hyperplanes is
		\[
		(2^r+1)s-1+t+r=m+r=n.
		\]
		
		Fix a control vector $\textbf{y}\in\Q_r$. On the slice $\{\textbf{y}\}\times\Q_m$, the main hyperplane \eqref{eq:main-hyperplane} becomes, by \eqref{eq:xor},
		\begin{equation}\label{eq:selected-level}
			S=a+s(1+[\textbf{b}\oplus \textbf{y}]).
		\end{equation}
		As $b$ ranges over $\Q_r$, the vector $\textbf{b}\oplus \textbf{y}$ also ranges over $\Q_r$. Consequently, for fixed $a$, the main hyperplanes select exactly the levels
		\[
		a+s,a+2s,\dots,a+2^rs.
		\]
		As \(a\) ranges from \(0\) to \(s-1\), these arithmetic progressions
		together contain every integer level in
		\(\{s,s+1,\ldots,(2^r+1)s-1\}\) exactly once. The lower layers cover $1,\dots,s-1$, and the upper layers cover
		\[
		(2^r+1)s,(2^r+1)s+1,\dots,(2^r+1)s+t-1=m.
		\]
		Thus, on every control slice, every positive Hamming-weight level $S=1,\dots,m$ is covered.
		
		It remains to cover level $S=0$ on the nonzero control slices. If $\textbf{y}\ne \textbf{0}$, then some $y_i=1$, and the corresponding cleanup hyperplane $C_i$ contains $(\textbf{y},\textbf{z})$ when $S(\textbf{z})=0$. At the origin, all main right-hand sides are positive, all lower and upper layers have positive right-hand side, and every cleanup equation reads $0=1$. Hence the origin is avoided.
	\end{proof}
	
	The proof shows a little more: apart from the cleanup hyperplanes, the construction partitions all vertices whose restriction to the last $m=n-r$ coordinates has positive Hamming weight.
	
	In order to prove the main theorem, we must determine the intersection sizes of the hyperplanes.
	
	\begin{proposition}[Intersection-size formulas]\label{prop:intersection-formulas}
		For the family $\mathcal H(n,r)$, the Boolean intersection sizes are as follows:
		\begin{align}
			\isect_n(H_{a,\textbf{b}})
			&=\sum_{h=1}^{2^r}\binom{m}{a+hs},
			&&0\le a<s,\ \textbf{b}\in\Q_r,\label{eq:main-size}\\
			\isect_n(L_j)
			&=2^r\binom{m}{j},
			&&1\le j<s,\label{eq:lower-size}\\
			\isect_n(U_\ell)
			&=2^r\binom{m}{(2^r+1)s+\ell},
			&&0\le\ell<t,\label{eq:upper-size}\\
			\isect_n(C_i)
			&=2^{r-1}(m+1),
			&&1\le i\le r.\label{eq:cleanup-size}
		\end{align}
		In particular,
		\begin{multline}\label{eq:general-upper}
			F(n)\le \max\Biggl\{
			\max_{0\le a<s}\sum_{h=1}^{2^r}\binom{m}{a+hs},\ 
			2^r\max_{1\le j<s}\binom{m}{j},
			2^r\max_{0\le\ell<t}\binom{m}{(2^r+1)s+\ell},\ 
			2^{r-1}(m+1)
			\Biggr\},
		\end{multline}
		where a maximum over an empty index set is omitted.
	\end{proposition}
	
	\begin{proof}
		For a fixed pair $(a,\textbf{b})$, the map $\textbf{y}\mapsto \textbf{b}\oplus \textbf{y}$ is a bijection of $\Q_r$. Thus \eqref{eq:selected-level} assumes the levels $a+hs$, $1\le h\le 2^r$, once each. Applying Corollary \ref{cor:control-tail} proves \eqref{eq:main-size}. The equations defining $L_j$ and $U_\ell$ are independent of the $2^r$ control vectors, giving \eqref{eq:lower-size} and \eqref{eq:upper-size}.
		
		For $C_i$, if $x_i=1$ then $S=0$, giving $2^{r-1}$ points. If $x_i=0$ then $S=1$, giving $2^{r-1}m$ points. This proves \eqref{eq:cleanup-size}.
	\end{proof}
	
	\begin{remark}[Total incidence]\label{rem:total-incidence}
		The main, lower, and upper hyperplanes cover every vertex with $S\ge1$ exactly once. Therefore
		\[
		\sum_{H\in\mathcal H(n,r)}|H\cap\Q_n|
		=2^n-2^r+r2^{r-1}(m+1).
		\]
		In the asymptotic regime used below, the correction to $2^n$ is subexponential. Thus not only the maximum intersection size, but also the total incidence, is asymptotically minimal.
	\end{remark}
	\subsection{Bounding the intersection sizes and setting the value of $r$}\label{sec:residues}

	The main intersection size in \eqref{eq:main-size} is controlled by sums of
	binomial coefficients over residue classes. For integers \(m,s\geq 1\) and
	\(a\in\mathbb Z\), define
	\[
	R_{m,s}(a)
	=
	\sum_{\substack{0\leq k\leq m\\ k\equiv a\pmod s}}
	\binom{m}{k}.
	\]
	The roots-of-unity identity used below is classical and is usually attributed
	to Ramus~\cite{Ramus1834}; see also~\cite{KonvalinaLiu1997}. We include a
	short proof in the Appendix because we require an estimate that is uniform when the modulus
	\(s\) grows with \(m\).
	
	\begin{lemma}[Ramus's identity and a uniform estimate]
		\label{lem:roots-unity}
		Let \(\omega=e^{2\pi i/s}\). Then
		\begin{equation}
			\label{eq:ramus-identity}
			R_{m,s}(a)
			=
			\frac{1}{s}
			\sum_{j=0}^{s-1}
			\omega^{-aj}(1+\omega^j)^m.
		\end{equation}
		Moreover,
		\begin{equation}
			\label{eq:fourier-error}
			\left|
			R_{m,s}(a)-\frac{2^m}{s}
			\right|
			\leq
			\frac{2^{m+1}}{s}
			\sum_{j=1}^{\lfloor s/2\rfloor}
			\exp\left(-\frac{2mj^2}{s^2}\right).
		\end{equation}
		Consequently, uniformly in \(a\),
		\begin{equation}
			\label{eq:roots-bound}
			R_{m,s}(a)
			\leq
			\frac{2^m}{s}
			+
			C\frac{2^m}{\sqrt m}
		\end{equation}
		for an absolute constant \(C\). Furthermore, if 
		$\frac{m}{s^2}\longrightarrow\infty,$
		then, uniformly in \(a\),
		\begin{equation}
			\label{eq:roots-asymptotic}
			R_{m,s}(a)
			=
			\left(1+o(1)\right)\frac{2^m}{s}.
		\end{equation}
	\end{lemma}

	
	
	We now choose the control dimension as a slowly growing function of $n$.
	
	\begin{proof}[Proof of \Cref{thm:main}]
		The lower bound follows from \eqref{eq:averaging-lb}, so it remains to prove the upper bound. Choose
		\begin{equation}\label{eq:r-choice}
			r=\left\lceil\log_2(\sqrt n\log n)\right\rceil,
		\end{equation}
		put $m=n-r$, and define $s,t$ as before:  $m+1=(2^r+1)s+t,\qquad 0\le t<2^r+1.$ For all sufficiently large $n$, we have \(2^r\leq m\), and hence \(s\geq1\). The choice \eqref{eq:r-choice} gives
		\begin{equation}\label{eq:p-range}
			\sqrt n\log n\le 2^r<2\sqrt n\log n.
		\end{equation}
		Since $t<2^r+1$,
		\[
		s=\frac{m+1-t}{2^r+1}\sim\frac{n}{2^r},
		\]
		and therefore
		\begin{equation}\label{eq:mixing}
			\frac{m}{s^2}\sim\frac{2^{2r}}{n}\ge(1+o(1))(\log n)^2\longrightarrow\infty.
		\end{equation}
		
		For a main hyperplane, \eqref{eq:main-size} is a partial sum of one residue class modulo $s$. By Lemma \ref{lem:roots-unity} and \eqref{eq:mixing}, uniformly in the pair $a,\textbf{b}$, we obtain 
		\begin{equation}\label{eq:main-asymptotic-1}
			\isect_n(H_{a,\textbf{b}})\le(1+o(1))\frac{2^m}{s}.
		\end{equation}
		Thus we have
		\[
		2^rs=m+1-s-t=n-r+1-s-t.
		\]
		By \eqref{eq:p-range},
		\[
		r=O(\log n),\qquad
		s=O\!\left(\frac{\sqrt n}{\log n}\right),\qquad
		t=O(\sqrt n\log n),
		\]
		so $2^rs=(1-o(1))n$. Consequently,
		\begin{equation}\label{eq:main-asymptotic-2}
			\frac{2^m}{s}=\frac{2^n}{2^rs}=(1+o(1))\frac{2^n}{n}.
		\end{equation}
		Together, \eqref{eq:main-asymptotic-1} and \eqref{eq:main-asymptotic-2} give the desired bound for the main hyperplanes.
		
		It remains to check that all auxiliary hyperplanes are negligible. For a lower layer, $j<s=O(\sqrt n/\log n)$, and the standard estimate
		\[
		\binom{m}{j}\le\left(\frac{em}{j}\right)^j
		\]
		shows
		\[
		2^r\binom{m}{j}=2^{o(n)}=o\!\left(\frac{2^n}{n}\right).
		\]
		For an upper layer, the complementary weight satisfies
		\[
		m-((2^r+1)s+\ell)=t-1-\ell<2^r+1=O(\sqrt n\log n).
		\]
		Hence the same estimate, applied after binomial symmetry, gives
		\[
		2^r\binom{m}{(2^r+1)s+\ell}=2^{o(n)}=o\!\left(\frac{2^n}{n}\right).
		\]
		Finally,
		\[
		\isect_n(C_i)=2^{r-1}(m+1)=O(n^{3/2}\log n)
		=o\!\left(\frac{2^n}{n}\right).
		\]
		Applying Proposition \ref{prop:intersection-formulas} completes the upper bound and hence the theorem.
	\end{proof}
	
	The proof allows a range of choices. The following formulation makes the transition from the Hamming scale to the averaging scale transparent.
	
	\begin{corollary}\label{cor:range}
		Suppose $r=r(n)$ satisfies $2^r=o(n)$, and let $s,t$ be as in the definition of $s$. Then the construction gives
		\[
		F(n)\le O\!\left(2^n\left(\frac1n+\frac{1}{2^r\sqrt n}\right)\right),
		\]
		provided the auxiliary lower and upper layers have intersection size $o(2^n/n)$. In particular, if $2^r=n^\alpha$ with $0<\alpha<1/2$, then
		\[
		F(n)=O\!\left(\frac{2^n}{n^{1/2+\alpha}}\right).
		\]
		When $2^r\gg\sqrt n$ and $m/s^2\to\infty$, the construction reaches $O(2^n/n)$.
	\end{corollary}
	
	\begin{proof}
		By \eqref{eq:roots-bound}, a main hyperplane has intersection size at most
		\[
		O\!\left(\frac{2^m}{s}+\frac{2^m}{\sqrt m}\right).
		\]
		Since $s\asymp n/2^r$ and $2^m=2^n/2^r$, the two terms are respectively $O(2^n/n)$ and $O(2^n/(2^r\sqrt n))$.
	\end{proof}

	\section{Concluding remarks}\label{sec:concluding}
	
	Our construction proves that  one can get a Alon--F\"uredi cover with small intersection sizes.   If $\mathcal H=\{H_1,\dots,H_n\}$ avoids the origin and covers $\Q_n^*$, and if $\ell_i$ is an affine linear form vanishing on $H_i$, then
	\[
	P(x)=\prod_{i=1}^n\ell_i(x)
	\]
	has degree $n$, vanishes on $\Q_n^*$, and is nonzero at the origin. The Alon--F\"uredi theorem asserts that the degree cannot be smaller. At equality, however, the polynomial method constrains the product $P$ but does not classify its factorization into affine linear forms. The coordinate cover, the Hamming-layer cover, and the present balanced covers provide very different factorizations with the same Boolean vanishing requirement. The parameter $F(n)$ measures the largest number of Boolean zeros lying on a single affine linear factor.

	The general construction is designed for asymptotics, but the same control-slice viewpoint is effective in small dimensions. The following table compares the averaging lower bound, the explicit upper bounds below, and the Hamming-layer upper bound. The explicit upper bounds are achieved using computer-aided search.
	
	\begin{table}[ht]
		\centering
		\begin{tabular}{@{}cccc@{}}
			\toprule
			$n$ & $\left\lceil(2^n-1)/n\right\rceil$ & explicit upper bound & Hamming bound\\
			\midrule
			2& 2&2&2\\
			3&3&3&3\\
			4&4&4&6\\
			5&7&7&10\\
			
			6 & 11 & 11 & 20\\
			7 & 19 & 21 & 35\\
			8 & 32 & 35 & 70\\
			9 & 57 & 71 & 126\\
			\bottomrule
		\end{tabular}
		\caption{Small-dimensional bounds for $F(n)$.}
	\end{table}
	
	\noindent \textbf{AI declaration.} The author found all the main ideas of the proof of the main theorem, however the general construction was inspired by the optimal solution found by ChatGPT 5.5. for small dimensional cases. It was also used in editing and finalizing the manuscript. The author  take full responsibility for the mathematical content.
	
	\noindent \textbf{Acknowledgement} The main problem of this paper was posed as an open problem for Emléktábla workshop 2026, Velence, Hungary.

	\newpage
	\appendix
	
	\section{Proof of Lemma \ref{lem:roots-unity}}
	
	\begin{proof}
		The roots-of-unity filter gives
		\[
		\frac{1}{s}\sum_{j=0}^{s-1}\omega^{j(k-a)}
		=
		\begin{cases}
			1, & k\equiv a\pmod s,\\
			0, & k\not\equiv a\pmod s.
		\end{cases}
		\]
		Therefore
		\begin{align*}
			R_{m,s}(a)
			&=
			\sum_{k=0}^m
			\binom{m}{k}
			\frac{1}{s}
			\sum_{j=0}^{s-1}\omega^{j(k-a)}
			=
			\frac{1}{s}
			\sum_{j=0}^{s-1}
			\omega^{-aj}
			\sum_{k=0}^m\binom{m}{k}\omega^{jk}
			=
			\frac{1}{s}
			\sum_{j=0}^{s-1}
			\omega^{-aj}(1+\omega^j)^m,
		\end{align*}
		which proves \eqref{eq:ramus-identity}.
		
		Since
		\[
		1+\omega^j
		=
		2e^{\pi i j/s}
		\cos\left(\frac{\pi j}{s}\right),
		\]
		the \(j=0\) term in \eqref{eq:ramus-identity} is \(2^m/s\). Hence
		\[
		\left|
		R_{m,s}(a)-\frac{2^m}{s}
		\right|
		\leq
		\frac{2^m}{s}
		\sum_{j=1}^{s-1}
		\left|
		\cos\left(\frac{\pi j}{s}\right)
		\right|^m.
		\]
		Using symmetry and
		\[
		\cos x
		\leq
		\exp\left(-\frac{2x^2}{\pi^2}\right),
		\qquad
		0\leq x\leq\frac{\pi}{2},
		\]
		we obtain
		\[
		\left|
		R_{m,s}(a)-\frac{2^m}{s}
		\right|
		\leq
		\frac{2^{m+1}}{s}
		\sum_{j=1}^{\lfloor s/2\rfloor}
		\exp\left(-\frac{2mj^2}{s^2}\right),
		\]
		which proves \eqref{eq:fourier-error}.
		
		Since the function
		\[
		x\longmapsto
		\exp\left(-\frac{2mx^2}{s^2}\right)
		\]
		is decreasing on \([0,\infty)\), comparison with the Gaussian integral gives
		\[
		\sum_{j=1}^{\infty}
		\exp\left(-\frac{2mj^2}{s^2}\right)
		\leq
		\int_0^\infty
		\exp\left(-\frac{2mx^2}{s^2}\right)\,dx
		=
		\sqrt{\frac{\pi}{8}}\frac{s}{\sqrt m}.
		\]
		Substituting this estimate into \eqref{eq:fourier-error} yields
		\[
		\left|
		R_{m,s}(a)-\frac{2^m}{s}
		\right|
		\leq
		\sqrt{\frac{\pi}{2}}\frac{2^m}{\sqrt m},
		\]
		uniformly in \(a\).
		Substitution into
		\eqref{eq:fourier-error} proves \eqref{eq:roots-bound}.
		
		Finally, if \(m/s^2\to\infty\), then
		\[
		\sum_{j=1}^{\infty}
		\exp\left(-2\frac{m}{s^2}j^2\right)
		\leq
		\sum_{j=1}^{\infty}
		\exp\left(-2\frac{m}{s^2}j\right)
		=
		o(1).
		\]
		Thus \eqref{eq:fourier-error} gives
		\[
		\left|
		R_{m,s}(a)-\frac{2^m}{s}
		\right|
		=
		o\left(\frac{2^m}{s}\right)
		\]
		uniformly in \(a\), proving \eqref{eq:roots-asymptotic}.
	\end{proof}
\end{document}